\documentclass[12pt,a4paper]{amsart}
\usepackage{amsfonts, amsmath, amssymb, amscd, amsthm, graphicx, xspace,url}%, %}
\usepackage{hyperref}
\usepackage[all]{xy}
\usepackage[usenames, dvipsnames]{color}

\newtheorem{thm}{Theorem}[section]
\newtheorem{cor}[thm]{Corollary}
\newtheorem{lemma}[thm]{Lemma}
\newtheorem{prop}[thm]{Proposition}

\theoremstyle{definition}

\theoremstyle{remark}
\newtheorem{rem}[thm]{Remark}

\newcommand{\Z}{\mathbb{Z}}
\newcommand{\N}{\mathbb{N}}

\newcommand{\VR}{\mathcal{VR}}
\newcommand{\AVR}{\mathcal{AVR}}
\newcommand{\X}{\mathcal{VCSH}}

\newcommand{\cd}{\mathrm{cd}}

\newcommand{\res}{\mathrm{res}}

\newcommand{\cores}{\mathrm{cor}}

\begin{document}
\title{Virtually compact special hyperbolic groups are conjugacy separable}

\author{Ashot Minasyan}
\address[Ashot Minasyan]{Mathematical Sciences,
University of Sout\-hampton, Highfield, Sout\-hampton, SO17 1BJ, United Kingdom.}
\email{aminasyan@gmail.com}

\author{Pavel Zalesskii}
\address[Pavel Zalesskii]{Departamento de Matem\'atica, Universidade de
Bras\'{i}lia, 70910-900 Bras\'{\i}lia-DF, Brazil.}
\email{pz@mat.unb.br}

\date{\footnotesize\today}



\begin{abstract}
We prove that any word hyperbolic group which is virtually compact special (in the sense of Haglund and Wise) is conjugacy separable. As a
consequence we deduce that all word hyperbolic Coxeter groups and many classical small cancellation groups are conjugacy separable.

To get the main result we establish a new criterion for showing that elements of prime order are conjugacy distinguished.
This criterion is of independent interest; its proof is based on a combination of discrete and profinite (co)homology theories.
\end{abstract}
\maketitle

\section{Introduction}
One of the main themes of Geometric Group Theory is the study of groups which act on non-positively curved spaces.
Two prominent classes of such groups is the class of hyperbolic groups (defined by Gromov in \cite{Gromov}) and the class of
(virtually) special groups (introduced by Haglund and Wise in \cite{H-W-1}).
The intersection of these two classes is  quite large and its elements, virtually special hyperbolic groups, have particularly nice properties.

Recall that a finitely generated group $G$ is said to be {\it hyperbolic} if its Cayley graph is a $\delta$-hyperbolic metric space, for some $\delta \ge 0$
(see, for example, \cite{Mih}). On the other hand, $G$ is \emph{virtually compact special}, if there is a finite index subgroup $H \leqslant G$, such that $H$ is isomorphic to the
fundamental group of a compact \emph{special cube complex}, whose hyperplanes satisfy certain combinatorial properties (see \cite[Sec. 3]{H-W-1}).

Since the original work of Haglund and Wise \cite{H-W-1}, many hyperbolic groups have been shown to be virtually special. For example, in the paper \cite{HW}
Haglund and Wise showed that hyperbolic Coxeter groups are virtually compact special. In \cite{Wise-qc-h} Wise proved the same for finitely generated
$1$-relator groups with torsion, while in \cite{Agol}
Agol showed this for fundamental groups of closed hyperbolic $3$-manifolds.
In fact, Agol \cite{Agol} proved that any hyperbolic group admitting a proper cocompact action on a CAT($0$)
cube complex is virtually compact special.

In this paper we study conjugacy separability of virtually compact special hyperbolic groups. Recall, that a group $G$ is
{\it conjugacy separable} if for arbitrary non-conjugate elements $x,y \in G$ there is a homomorphism
from $G$ to a finite group $F$ such that the images of $x$ and $y$
are not conjugate in $F$. Conjugacy separability can be regarded as an algebraic analogue of solvability of the conjugacy problem in a group and has a number of applications.
Most prominently it is used in proving residual finiteness of outer automorphism groups (see, for example, the discussion in \cite[Sec. 2]{M-RAAG}).

Conjugacy separability is usually not easy to show, and, until recently, only a few classes of groups were known to satisfy it: virtually free groups \cite{Dyer},
virtually surface groups \cite{Martino} and virtually polycyclic groups \cite{Form, Rem}. Note that in general conjugacy separability does not pass to finite
index overgroups \cite{Gor} or to finite index subgroups \cite{M-M}, therefore the adjective ``virtually'' is important.

A group $G$ is said to be {\it hereditarily conjugacy separable} if every finite index subgroup of $G$ is conjugacy separable.
In \cite{M-RAAG} the first author showed that right angled Artin groups are hereditarily conjugacy separable.
This result was subsequently used to prove conjugacy separability of Bianchi groups \cite{C-Z-Bianchi}, $1$-relator groups with torsion \cite{M-Z_1-rel} and fundamental groups
of compact $3$-manifolds \cite{H-W-Z}. In fact, in \cite{M-RAAG} it was shown that any virtually compact special group $G$ contains a conjugacy separable
subgroup of finite index. But it is still unclear whether such $G$ must necessarily be conjugacy separable itself.
In the present paper we prove this in the case when $G$ is hyperbolic:

\begin{thm}\label{thm:main} Any virtually compact special hyperbolic group is hereditarily conjugacy separable.
\end{thm}

Conjugacy separability of torsion-free virtually compact special hyperbolic groups was proved in \cite[Cor. 9.11]{M-RAAG}, so the actual novelty of
Theorem \ref{thm:main} is in handling groups with torsion.
In view of Agol's result \cite[Thm. 1.1]{Agol}, the above theorem shows that every hyperbolic group, admitting a proper
cocompact action on a CAT($0$) cube complex, is hereditarily conjugacy separable. This gives an abundance of new examples of (hereditarily) conjugacy separable groups,
some of which we mention in corollaries below.

For any Coxeter group $W$, Niblo and Reeves \cite{Niblo-Reeves} constructed a cube complex $\mathcal C$ on which $W$ acts properly,
and proved that the quotient complex
$\mathcal{X}=W \backslash \mathcal{C}$ is compact if $W$ is hyperbolic.
It follows that any hyperbolic Coxeter group is virtually compact special (originally this is due to Haglund and Wise \cite{HW}), hence
we can use Theorem \ref{thm:main} to deduce:

\begin{cor}\label{cor:Cox} Any hyperbolic Coxeter group is hereditarily conjugacy separable.
\end{cor}

Note that conjugacy separability of hyperbolic even Coxeter groups was proved in \cite{C-M}.

Another family of hyperbolic virtually compact special groups is given by groups with finite small cancellation presentations.
Indeed, in \cite{Wise_sm-c} Wise proved that many classical small
cancellation groups, including $C'(1/6)$ and $C'(1/4)-T(4)$ groups, act properly and cocompactly on CAT($0$) cube complexes.
It is well-known that such groups are hyperbolic, so Agol's result \cite[Thm. 1.1]{Agol} applies and, together with Theorem \ref{thm:main}, it yields

\begin{cor} Let $G$ be a group with a finite $C'(1/6)$ or $C'(1/4)-T(4)$ presentation. Then $G$ is hereditarily conjugacy separable.
\end{cor}

Finally, Theorem \ref{thm:main} implies that any group acting properly and cocompactly on the hyperbolic $3$-space is hereditarily
conjugacy separable, because fundamental groups of closed hyperbolic $3$-manifolds are
virtually compact special by a combination of results of Bergeron and Wise \cite{B-W} and Agol \cite{Agol}. Thus we obtain the following statement:

\begin{cor} Any uniform lattice in $PSL_2(\mathbb{C})$ is
hereditarily conjugacy separable.
\end{cor}

The above corollary could also be proved by combining results of Chagas and the second author
\cite[Thm. 2.5 or Thm. 2.7]{C-Z-Bianchi} with a different theorem of Agol from \cite{Agol}, claiming that
closed hyperbolic $3$-manifolds are virtually fibered. %(see Theorem 9.2 in )

Let us now say a few words about the proof of Theorem \ref{thm:main}.
One of the main difficulties in it is to separate conjugacy classes of
torsion elements in a finite quotient. To this end we come up with a new approach (see Proposition \ref{prop:prime_order}) which employs (co)homological methods and is
based on a result of K.S. Brown \cite{Brown} allowing one to distinguish conjugacy classes of elements of prime order using group cohomology. In particular we obtain the following
quite general result.

\begin{thm}\label{thm:prime_order_cd} Let $G$ be a  residually finite group with ${\rm vcd}(G)<\infty$.
If $G$ is cohomologically good then every element of prime order is conjugacy distinguished in $G$.
\end{thm}

Recall that a residually finite group $G$ is {\it cohomologically good}, if the inclusion of $G$ in its profinite completion induces an
isomorphism on cohomology with finite coefficients. An element $g \in G$ is said to be \emph{conjugacy distinguished} if the conjugacy class $g^G$ is closed in the
profinite topology on $G$ (thus $G$ is conjugacy separable if and only if each $g \in G$ is conjugacy distinguished).
The claim of Theorem \ref{thm:prime_order_cd} can be restated by saying that two non-conjugate elements of prime order in $G$ are not conjugate in the profinite completion $\widehat G$; in other words, the embedding of $G$ in $\widehat G$ induces an injective map
on the sets of conjugacy classes of elements of prime order in $G$ and in $\widehat G$. In Corollary \ref{cor:bij_cc} we prove that if, additionally,
$G$ is finitely generated then this map is actually a bijection (in particular, every element of prime order in $\widehat G$ is conjugate to some element in $G$).

To prove Theorem \ref{thm:main} for a hyperbolic virtually compact special group $G$, we first show that $G$ is cohomologically good by proving that this property
is stable under
virtual retractions (Lemma~\ref{lem:retract-good}), and combining this with some results from \cite{H-W-1,GJZ,Lor} (our argument actually does not
make use of the hyperbolicity of $G$ and works, more generally, for almost virtual retracts of right angled Artin groups -- see Proposition \ref{prop:VCSH-good}).
It follows that Theorem \ref{thm:prime_order_cd} can be applied to separate the conjugacy classes of
elements of prime order in $G$. After this we prove that every torsion element of $G$ is conjugacy distinguished essentially by
induction on its order.

\medskip
\noindent{\bf Acknowledgements.} The authors would like to thank Marco Boggi, Fr\'ed\'eric Haglund, Dessislava Kochloukova, Ian Leary and Nansen Petrosyan for enlightening discussions.
The second author was  supported by Capes and CNPq.

\section{Preliminaries}
\subsection{Notation}
Given a group $G$, its subgroups $K,H$ and an element $g \in G$, we will write $C_H(g)=\{h \in H \mid hgh^{-1}=g\}$ to denote the \emph{centralizer} of $g$ in $H$, and
$N_H(K)=\{h \in H \mid hKh^{-1}=K\}$ to denote the \emph{normalizer} of $K$ in $H$.

\subsection{Hyperbolic groups and quasiconvex subgroups}
Recall that a geodesic metric space $Y$ is (Gromov) \emph{hyperbolic} if there exists a constant $\delta \ge 0$ such that for any geodesic triangle $\Delta$ in $Y$,
any side of $\Delta$ is contained in the closed $\delta$-neighborhood of the union of the other sides (cf. \cite{Mih}). A subset $Z \subseteq Y$ is \emph{quasiconvex} if
there is $\varepsilon \ge 0$ such that for any two points $z_1,z_2 \in Z$, any geodesic joining these points is contained in the closed $\varepsilon$-neighborhood of $Z$.

If $G$ is a group generated by a finite set $\mathcal A \subseteq G$, then $G$ is said to be \emph{hyperbolic} if its Cayley graph $\Gamma(G,\mathcal{A})$ is a hyperbolic metric space.
Similarly, a subset $S \subseteq G$ is  \emph{quasiconvex} if it is quasiconvex when considered as a subset of $\Gamma(G,\mathcal{A})$.

Quasiconvex subgroups are very important in the study of hyperbolic groups. Such subgroups are themselves hyperbolic and are quasi-isometrically embedded in $G$ (see \cite{Mih}).
Basic examples of quasiconvex subgroups in hyperbolic groups are centralizers of elements (see \cite[Ch. III.$\Gamma$, Prop. 3.9]{B-H}); this fact
will be important for our argument below.

\subsection{Right angled Artin groups} A right angled Artin group is a group which can be given by a finite
presentation, where the only defining relators are commutators of the generators. To construct such a group, one usually starts with a finite simplicial graph $\Gamma$
with vertex set $V$ and edge set $E$. One then defines the corresponding \emph{right angled Artin group} $A=A(\Gamma)$ by the following presentation:
$$A= \langle V \,\|\, [u,v]=1, \mbox{ whenever } (u,v)\in E \rangle,$$
where $[u,v]=uvu^{-1}v^{-1}$ is the commutator of $u$ and $v$.

For any subset $S \subseteq V$, the subgroup $A_S=\langle S \rangle \leqslant A$ is said to be a \emph{full subgroup} of $A$. It is easy to see that $A_S$ is itself a
right angled Artin group corresponding to the full subgraph $\Gamma_S$ of $\Gamma$, induced by the vertices from $S$. Moreover, $A_S$ is a retract of $A$ --
see \cite[Sec. 6]{M-RAAG}.

Recall that a subgroup $H$, of a group $G$, is a \emph{virtual retract} if $H$ is a retract of some finite index subgroup $K \leqslant G$.
In other words, $H \subseteq K$ and there is a homomorphism $\rho: K \to H$ such that $\rho(K)=H$ and $\rho|_H=\mathrm{id}_H$.

Let $\VR$ denote the class of all groups which are virtual retracts of
finitely generated right angled Artin groups, and let $\AVR$ be the class consisting of all groups $G$ such that $G$ has a finite index subgroup from $\VR$.
We are interested in these specific classes of groups because of the following two results:
in \cite{H-W-1} Haglund and Wise proved that any virtually compact special group $G$ belongs to the class $\AVR$, and in \cite{M-RAAG} the first author showed that
any group $H \in \VR$ is hereditarily conjugacy separable.

\subsection{Profinite topology} The \emph{profinite topology} on a group $G$ is defined by taking finite index subgroups as a basis of neighborhoods of the identity element.
This topology is Hausdorff, i.e., $\{1\}$ is a closed subset of $G$,  if and only if the group $G$ is residually finite. In the latter case,
$G$ embeds in its profinite completion, $\widehat G$, and the profinite topology on $G$ is precisely the restriction of the natural topology of $\widehat G$ to $G$.

A subset $S \subseteq G$ is said to be \emph{separable}
if it is closed in the profinite topology on $G$.
Thus an element $x \in G$ is  conjugacy distinguished if its conjugacy class $x^G=\{gxg^{-1} \mid g \in G\}$ is separable in $G$.
It is not difficult to see that the latter is equivalent to the property that for any element $y \in G$, which is not conjugate to $x$, there is a finite group
$F$ and a homomorphism $\phi:G \to F$, such that $\phi(y)$ is not conjugate to $\phi(x)$ in $F$. It follows that
$G$ is conjugacy separable if and only if all of its elements are conjugacy distinguished.

\subsection{Criteria for conjugacy separability}
The next standard observation will be useful (cf. \cite[Lemma 7.2]{M-M}):
\begin{lemma} \label{lem:sep_in_fi_sbgp} Let $K$ be a subgroup of finite index in a group $G$ and let $x \in K$.
If $x$ is conjugacy distinguished in $K$ then $x$ is conjugacy distinguished in $G$.
\end{lemma}

The following criterion was discovered by Chagas and the second author in \cite{C-Z-Bianchi}:
\begin{prop}[{\cite[Prop. 2.1]{C-Z-Bianchi}}]\label{prop:C-Z_crit} Let $H$ be a normal subgroup of index $m \in \N$ in a group $G$ and let $x \in G$ be any element. Suppose that $H$ is hereditarily conjugacy separable
and the centralizer $C_G(x^m)$, of $x^m \in H$,  satisfies the following conditions:
\begin{itemize}
    \item[(i)] $x$ is conjugacy distinguished in $C_G(x^m)$;
    \item[(ii)] each finite index subgroup of  $C_G(x^m)$ is separable in $G$.
\end{itemize}
Then $x$ is conjugacy distinguished in $G$.
\end{prop}

Note that the original condition (i) from \cite[Prop. 2.1]{C-Z-Bianchi} required $C_G(x^m)$ to be conjugacy separable, however, it is easy to see that
the proof (see also \cite[Prop. 2.2]{C-M} for an alternative argument) only uses the weaker assumption that $x$ is conjugacy distinguished in $C_G(x^m)$.

\subsection{Profinite topology on virtually compact special groups}
Let $\X$ denote the class of all virtually compact special hyperbolic groups.

\begin{rem}\label{rem:fi_in_X} The class $\X$ is closed under taking finite index subgroups and overgroups.
\end{rem}

Indeed, it is immediate from the definitions that a finite index subgroup/overgroup of a virtually compact special group is still virtually compact special.
On the other hand, it is well-known that a group is hyperbolic if and only if a finite index subgroup is hyperbolic (for instance,
this follows from the fact that hyperbolicity is invariant under quasi-isometries -- see \cite[Ch. III.H, Thm. 1.9]{B-H}).

The next statement easily follows from the work of Haglund and Wise in \cite{H-W-1}.

\begin{lemma}\label{lem:hyp_in_X} Suppose that $G \in \X$ and $g \in G$. Then
\begin{itemize}
  \item[(a)] the centralizer $C_G(g)$ also belongs to $\X$;
 \item[(b)] every finite index subgroup of $C_G(g)$ is separable in $G$.
\end{itemize}
\end{lemma}

\begin{proof} Fix some finite generating set $\mathcal A$ of $G$. Since the group $G$ is hyperbolic, it is well-known that centralizers of elements in $G$ are quasiconvex
(see, for example, \cite[Ch. III.$\Gamma$, Prop. 3.9]{B-H}). Hence $C_G(g)$ is quasiconvex, so it is also hyperbolic (cf. \cite[Lemma 3.8]{Mih}).
In \cite[Cor. 7.8]{H-W-1} Haglund and Wise proved that
any quasiconvex subgroup of $G$ is virtually compact special, thus (a) is proved.

To prove (b), note that every finite index subgroup $N \leqslant C_G(g)$ is also quasiconvex (because there is a constant $c \ge 0$ such that every element of
$C_G(g)$ is at distance no more than
$c$ from an element of $N$ in the Cayley graph $\Gamma(G,\mathcal{A})$). Therefore $N$ is separable in $G$ by \cite[Cor. 7.4 and Lemma 7.5]{H-W-1}.
\end{proof}

\begin{lemma}\label{lem:hcs_sbgp} Any virtually compact special group $G$ has a finite index normal subgroup $H \lhd G$ such that $H \in \VR$,
$H$ is torsion-free and hereditarily conjugacy separable.
\end{lemma}

\begin{proof} In \cite{H-W-1} Haglund and Wise proved that every virtually compact special group $G$  has
a finite index normal subgroup $H \lhd G$ such that $H \in \VR$. Now, $H$ is torsion-free as right angled Artin
groups are torsion-free, and $H$ is hereditarily conjugacy separable by \cite[Cor. 2.1]{M-RAAG}.
\end{proof}

\section{Cohomological goodness and its applications to conjugacy separability}

Recall that a  group $G$ is {cohomologically good}, if the natural
embedding $G\hookrightarrow \widehat{G}$, of the group in its
profinite completion, induces an isomorphism on cohomology with
finite coefficients. This notion was originally introduced by Serre in \cite[Exercises in Sec. I.2.6]{Serre}.

Cohomological goodness of residually finite groups behaves nicely under certain  free constructions and is stable under group commensurability (see \cite{GJZ,Lor}).
We begin this section with proving another useful permanence property:

\begin{lemma} \label{lem:retract-good} Suppose that $G$ is a residually finite cohomologically good group and $H$ is a virtual retract of
$G$. Then $H$ is cohomologically good.
\end{lemma}

\begin{proof} Since the cohomological goodness passes to subgroups
of finite index (see  \cite[Lemma~3.2]{GJZ}), we may assume that $H$ is a retract of $G$. Let
$f: G\to H$ be a retraction.  Then the profinite topology on $G$ induces the full profinite topology on $H$ (see, for example, \cite[Lemma 3.1.5]{RZ}), hence
the natural embedding $i: H \to G$ induces an injective continuous map $\widehat i: \widehat H \to \widehat G$ (cf. \cite[Lemma 3.2.6]{RZ}).
Therefore, the functorial property of profinite completions shows that the retraction $f$ induces a retraction $\widehat f: \widehat G \to \widehat H$, giving
rise to the following commutative diagram, where the vertical maps are the natural embeddings of the residually finite groups in their profinite completions:

\begin{equation}\label{eq:rest_diag} \xymatrix{ \widehat H\ar@<0.5ex>[rr]^{\widehat i}&& \widehat{G}\ar@<0.5ex>[ll]^{\widehat f}\\
&&&\\
 H\ar[uu] \ar @<0.5ex>[rr]^{i}    &&
G\ar[uu] \ar@<0.5ex>[ll]^{f}\cr}
\end{equation}

If $M$ is a finite $H$-module, we can turn it into a $G$-module by letting the
kernel of $f$ act trivially on $M$.  Then for any $n \in \N \cup \{0\}$, \eqref{eq:rest_diag} induces the following
commutative diagram of cohomology groups:

$$\xymatrix{H^n(\widehat H,M)\ar
@<0.5ex>[rr]^{{\widehat{f}}^*}\ar[dd]_{\res_H^{\widehat
H}}&& H^n(\widehat{G},M)\ar[dd]^{\res_G^{\widehat G}}\ar@<0.5ex>[ll]^{{\widehat{i}^*}}\\
&&&\\
 H^n(H,M)\ar @<0.5ex>[rr]^{f^*}    &&
H^n(G,M)\ar@<0.5ex>[ll]^{i^*}\cr}$$

Since $f \circ i =\mathrm{id}_H$ and $\widehat f \circ \widehat i =\mathrm{id}_{\widehat H}$, we can deduce that $i^* \circ f^*$ and $\widehat{i}^*\circ \widehat{f}^*$
are identity maps on $H^n(H,M)$ and $H^n(\widehat{H},M)$ respectively.
In particular, the map $\widehat{f}^*$ is injective and the map $i^*$ is surjective.

Since $G$ is cohomologically good the right vertical arrow is a
bijection and we need to show that so is the left vertical arrow.
To see the injectivity, pick an element $h\in H^n(\widehat H,M)$.
Then $\left(f^* \circ \res_H^{\widehat H}\right)(h)=\left(\res_G^{\widehat G} \circ \widehat{f}^*\right)(h)$, implying that $h=0$ if
$\res_H^{\widehat H}(h)=0$.

For surjectivity, observe that $i^*\circ \res_G^{\widehat G}=\res_H^{\widehat H}\circ \widehat{i}^*$ and the map on the left-hand side is surjective,
hence $\res_H^{\widehat H}$ must also be surjective.

Thus $\res_H^{\widehat H}$ is an isomorphism, as required.
\end{proof}

The next statement establishes a connection between cohomological goodness and separability of conjugacy classes of elements of prime order.

\begin{prop}\label{prop:prime_order} Let  $G$ be a
residually finite cohomologically good group of  finite virtual cohomological
dimension. Suppose that  $G$ splits as a semidirect product $G=H \rtimes \langle x \rangle$, where $H\lhd G$ is torsion-free and $x
\in G$ has prime order $p$. Then
the natural embedding of $G$ in $\widehat G$ induces an injective map between
the conjugacy classes of finite subgroups in $G$ and in $\widehat G$.
\end{prop}

\begin{proof} Fix any integer $n > {\rm vcd}(G)$.
Let $I$ [respectively, $\hat I$] denote the set of conjugacy classes of subgroups of order $p$ in $G$ [respectively, in $\widehat G$].
For every conjugacy class $\alpha \in I$ choose any subgroup $C_\alpha$, of order $p$, representing it in $G$.
Since all elementary abelian $p$-subgroups of $G$ have rank at most $1$ (as $G=H \rtimes \langle x \rangle$ and $H$ is torsion-free),
we can apply a classical result of Brown (cf. Cor. 7.4 and the Remark below it in Ch. X of \cite{Brown}),
claiming that there is a canonical isomorphism
\begin{equation} \label{eq:Brown} \eta: H^n(G,\Z/p)\to \prod_{\alpha\in I} \, H^n(N_{G}(C_\alpha),\Z/p).\end{equation}

Denote $N_\alpha=N_G(C_\alpha)$, $\alpha \in I$. The above isomorphism $\eta$ can be defined as follows: for each  $\alpha\in I$,
the inclusion $N_\alpha \hookrightarrow G$ induces
the restriction homomorphism $\res_{N_\alpha}^G: H^n(G,\Z/p) \to H^n(N_\alpha,\Z/p)$, and $\eta=\prod_{\alpha \in I} \, \res_{N_\alpha}^G$ is the
corresponding diagonal map.

For our purposes, it is actually more convenient to work with homology instead of cohomology. For each
$\alpha \in I$, the inclusion $N_\alpha \hookrightarrow G$ induces the corestriction homomorphism $\cores_{N_\alpha}^G: H_n(N_\alpha,\Z/p) \to H_n(G,\Z/p)$.
This gives a natural homomorphism
\begin{equation} \label{eq:Brown-homol} \varphi: \bigoplus_{\alpha \in I} H_n(N_\alpha,\Z/p) \to H_n(G,\Z/p),
\end{equation}
defined by the property that the restriction of $\varphi$ to each direct summand $H_n(N_\alpha,\Z/p)$ is the map $\cores_{N_\alpha}^G$.

Since $\Z/p$ is a field, the contravariant functor ${\rm Hom}_{\Z/p}(-,\Z/p)$ induces a natural isomorphism between
${\rm Hom}_{\Z/p} (H_n(F,\Z/p), \Z/p)$ and $H^n(F,\Z/p)$ for any group $F$
(for example by the Universal Coefficient Theorem, cf. \cite[Sec. 3.1, pp. 196-197]{Hatcher}). 
Applying this functor to \eqref{eq:Brown-homol} gives the map $\eta$ from \eqref{eq:Brown}.

If the map $\varphi$ was not injective then we would have a short exact sequence
$$\{0\} \to K \to  \bigoplus_{\alpha \in I} H_n(N_\alpha,\Z/p) \stackrel{\varphi}{\to} H_n(G,\Z/p) \to \{0\},$$
where $K$ is a non-trivial vector space over $\Z/p$. Since $\Z/p$ is a field,
the functor ${\rm Hom}_{\Z/p}(-,\Z/p)$  is exact, so it would give a short exact sequence
$$\{0\} \to H^n(G,\Z/p) \stackrel{\eta}{\to} \prod_{\alpha \in I} H^n(N_\alpha,\Z/p) \to {\rm Hom}_{\Z/p}(K,\Z/p) \to \{0\}.$$
The latter would contradict the fact that $\eta$ is surjective, as ${\rm Hom}_{\Z/p}(K,\Z/p) \neq \{0\}$. Therefore $\varphi$ is injective.
A similar argument shows that $\varphi$ is also surjective, as $\eta$ is injective. Hence the homomorphism $\varphi$ in \eqref{eq:Brown-homol} is an isomorphism.

In particular, we see that if $\alpha_1$ and $\alpha_2$ are distinct elements of $I$ then
\begin{equation}\label{eq:diff_phi}
\varphi(H_n(N_{\alpha_1},\Z/p)) \cap \varphi(H_n(N_{\alpha_2},\Z/p))=\{0\} \mbox{ in } H_n(G,\Z/p).
\end{equation}
By the assumptions, for each $k=1,2$, $G=H \rtimes C_{\alpha_k}$, i.e., $G$  retracts onto $C_{\alpha_k}$. Therefore $N_{\alpha_k}$ also retracts onto $C_{\alpha_k}$, and hence
the corestriction homomorphism $\cores^{N_{\alpha_k}}_{C_{\alpha_k}}: H_n(C_{\alpha_k},\Z/p)\to H_n(N_{\alpha_k},\Z/p)$ is injective.
Since $H_n(C_{\alpha_k},\Z/p) \neq \{0\}$ for $k=1,2$ (as $C_{\alpha_k} \cong \Z/p$), \eqref{eq:diff_phi} shows that
the natural images of $H_n(C_{\alpha_1},\Z/p)$ and $H_n(C_{\alpha_2},\Z/p)$ in
$H_n(G,\Z/p)$ must be distinct.

Now, arguing by contradiction, assume that there exist  distinct $\alpha_1, \alpha_2 \in I$ such that $C_{\alpha_1}$ is conjugate to $C_{\alpha_2}$ in $\widehat G$.
We have the following commutative diagram coming from the natural inclusions:
\begin{equation}\label{eq:comm_diag}
\xymatrix{& \widehat{G} \\
C_{\alpha_1}  \ar[ru] \ar[r] & G \ar[u]& \ar[l] \ar[lu] C_{\alpha_2}\cr}.
\end{equation}

Since $C_{\alpha_k}$ is a closed subgroup of $\widehat G$, $k=1,2$, and $G$ is dense in $\widehat G$,
this diagram induces the following commutative diagram of cohomology groups (for the vertical and diagonal arrows see \cite[Sec. I.2.4 and Exercise 1) in Sec. I.2.6]{Serre}):

\begin{equation}\label{eq:comm_diag-cohomol}
\xymatrix{& H^n(\widehat{G},\Z/p) \ar[ld]_{\res_{C_{\alpha_1}}^{\widehat G}}\ar[d]^{\res_G^{\widehat G}}\ar[rd]^{\res_{C_{\alpha_2}}^{\widehat G}}\\
H^n(C_{\alpha_1},\Z/p)    & H^n(G,\Z/p)\ar[l]^-{\res_{C_{\alpha_1}}^G} \ar[r]_-{\res_{C_{\alpha_k}}^G} &   H^n(C_{\alpha_2},\Z/p)\cr},
\end{equation}
where $\res_G^{\widehat G}$ is an isomorphism by
cohomological goodness of $G$.

Let us apply the ${\rm Hom}_{\Z/p}(-,\Z/p)$ functor to the diagram \eqref{eq:comm_diag-cohomol}.
Pontryagin duality between cohomology and homology of profinite groups (see \cite[Prop. 6.3.6]{RZ}) says that ${\rm Hom}_{\Z/p}(H^n(\widehat G,\Z/p),\Z/p)$ is naturally isomorphic
to $H_n(\widehat G,\Z/p)$. On the other hand, for the discrete group $G$, ${\rm Hom}_{\Z/p}(H^n(G,\Z/p),\Z/p)$ may not be, in general, isomorphic to $H_n(G,\Z/p)$.
However, since ${\rm Hom}_{\Z/p} (H_n(G,\Z/p), \Z/p) \cong H^n(G,\Z/p)$ (as observed above),
the space ${\rm Hom}_{\Z/p}(H^n(G,\Z/p),\Z/p)$ can be thought of as the double dual of $H_n(G,\Z/p)$.
Since there is always a canonical embedding of a vector space into its double dual, we obtain an injective homomorphism $\rho: H_n(G,\Z/p)\to H_n(\widehat G,\Z/p)$, which
fits into the following commutative diagram:
\begin{equation}\label{eq:comm_diag-homol}
\xymatrix{& H_n(\widehat{G},\Z/p) \\
H_n(C_{\alpha_1},\Z/p)  \ar[ru]^{\hat\tau_1} \ar[r]^{\tau_1} & H_n(G,\Z/p) \ar[u]_\rho & \ar[l]_{\tau_2} \ar[lu]_{\hat\tau_2} H_n(C_{\alpha_2},\Z/p)\cr},
\end{equation}
where $H_n(\widehat{G},\Z/p)$ is the profinite homology of $\widehat G$,
$\tau_k=\cores_{C_{\alpha_k}}^G$ and $\hat\tau_k=\cores_{C_{\alpha_k}}^{\widehat G}$, $k=1,2$.

By the assumption, there exists $g \in \widehat G$ such that $C_{\alpha_2}=g C_{\alpha_1} g^{-1}$. Hence we have
$$ \xymatrix{\widehat{G}\ar[d]^{i_g}  &\ar@{_{(}->}[l] C_{\alpha_1} \ar[d]^{i_g|_{C_{\alpha_1}}} \\
\widehat{G} &\ar@{_{(}->}[l] C_{\alpha_2}},$$
where $i_g: \widehat{G} \to \widehat{G}$ is the inner automorphism of $\widehat{G}$ given by $i_g(h)=ghg^{-1}$, for all $h \in \widehat{G}$,
and $i_g|_{C_{\alpha_1}}:C_{\alpha_1} \to C_{\alpha_2}$ is its restriction to $C_{\alpha_1}$.
This leads to the following commutative diagram between the corresponding homology groups:
$$ \xymatrix{H_n(\widehat{G},\Z/p)\ar[d]^{\mathrm{id}} & \ar[l]_{\hat\tau_1}H_n(C_{\alpha_1},\Z/p) \ar[d]^{\cong} \\
H_n(\widehat G,\Z/p) & \ar[l]_{\hat\tau_2} H_n(C_{\alpha_2},\Z/p)} ,$$
Note that the left vertical map is the identity on $H_n(\widehat G,\Z/p)$, as it is induced by an inner
automorphism of $\widehat G$ (this is easy to prove directly, or one can use
\cite[Execise 1) in Sec. I.2.5]{Serre} and apply the Pontryagin duality between $H^n$ and $H_n$).
Therefore we can conclude that
$\hat\tau_1(H_n(C_{\alpha_1},\Z/p))=\hat\tau_2(H_n(C_{\alpha_2},\Z/p))$ in $H_n(\widehat G,\Z/p)$.
Thus, in view of injectivity of the map $\rho$ from \eqref{eq:comm_diag-homol}, in $H_n(G,\Z/p)$ we must have that
$\tau_1(H_n(C_{\alpha_1},\Z/p))=\tau_2(H_n(C_{\alpha_2},\Z/p))$. The latter gives a contradiction with the property that the natural images
of $H_n(C_{\alpha_1},\Z/p)$ and $H_n(C_{\alpha_2},\Z/p)$ in $H_n(G,\Z/p)$ are distinct,
which was proved above as a consequence of the fact that the map $\varphi$ in \eqref{eq:Brown-homol} is injective.

Therefore, $C_{\alpha_1}$ cannot be conjugate to $C_{\alpha_2}$ in $\widehat G$ if $\alpha_1 \neq \alpha_2$ in $I$. This means that
the inclusion $G \hookrightarrow \widehat G$ induces an injective map from $I$ to $\hat I$, as required.
\end{proof}

We are now ready to prove Theorem \ref{thm:prime_order_cd}, stated in the Introduction.

\begin{proof}[Proof of Theorem \ref{thm:prime_order_cd}] Let $p$ be a prime and let $x$ be an element of order $p$ in $G$.
By the assumptions there exists a torsion-free normal subgroup $H \lhd G$, which has finite index in $G$.
Denote $G_1=H \langle x \rangle \leqslant G$. Clearly $G_1$ has finite index in $G$, and $G_1\cong H \rtimes \langle x \rangle$.
Therefore $G_1$ is residually finite and ${\rm vcd}(G_1)={\rm vcd}(G)<\infty$.
Moreover, $G_1$ is cohomologically good since this property passes to finite index subgroups and overgroups (see \cite[Lemma 3.2]{GJZ}).
Thus  the group $G_1$ satisfies all the assumptions of Proposition \ref{prop:prime_order}.

Consider any element $y \in G_1$, which is not conjugate to $x$. If $y$ and $x$ have different orders, then, using residual finiteness of $G_1$,
we can find a finite quotient $M$, of $G_1$, where the images of $y$ and $x$ still have different orders, and hence they will not be conjugate in $M$.
Therefore in this case $M$ will be a
finite quotient of $G_1$ distinguishing the conjugacy classes of $y$ and $x$.

So, now we can suppose that $y$ also has order $p$. If $\langle y \rangle$ is not conjugate to $\langle x \rangle$ in $G_1$, then, by Proposition \ref{prop:prime_order},
these subgroups are also not conjugate in $\widehat G_1$.
Hence $y$ is not conjugate to $x$ in $\widehat G_1$, i.e., $y \notin x^{\widehat{G}_1}$. Now, the conjugacy class $x^{\widehat{G}_1}$ is closed in
$\widehat G_1$, as $\widehat G_1$ is
compact, so $x^{\widehat G_1} \cap G_1$ is a separable subset of $G_1$ which contains $x^{G_1}$ but avoids $y$. It follows that
there is a finite quotient of $G_1$ distinguishing the conjugacy classes of
$x$ and $y$.

Thus we can further assume that $\langle y \rangle$ is conjugate to $\langle x \rangle$ in $G_1$. Then $hyh^{-1}=z$ for some $h \in G_1$ and some
$z \in \langle x \rangle$. Note that $z \neq x$
as $y$ is not conjugate to $x$ in $G_1$, by our assumption.
Consequently, $z=\xi(z) \neq \xi(x)=x$, where $\xi:G_1 \to \langle x \rangle$ is the natural retraction (coming from the semidirect product decomposition of $G_1$).
Since the group $\langle x \rangle$ is abelian, we can conclude that $\xi(y)=\xi(z)$ is not conjugate to $\xi(x)$ in it, so $\langle x \rangle$
is a finite quotient of $G_1$ distinguishing the conjugacy classes of $x$ and $y$.

Thus we have considered all possibilities, showing that $x$ is conjugacy distinguished in $G_1$. It remains to apply Lemma \ref{lem:sep_in_fi_sbgp} to conclude that
$x$ is conjugacy distinguished in $G$, as required.
\end{proof}

Proposition \ref{prop:prime_order} shows that, under its assumptions, the natural inclusion $G \to \widehat{G}$ induces an injective map between
the conjugacy classes of prime order subgroups in $G$ and in $\widehat G$. To complement this, we will now show this map is also surjective, provided $G$ has finitely many conjugacy classes of elements of prime order (the latter will be satisfied if $G$ is finitely generated -- see Corollary \ref{cor:bij_cc} below).

\begin{lemma}\label{lem:f-cd-t-f} Suppose that $H$ is a cohomologically good group with ${\rm cd}(H)=n<\infty$. Then $\cd(\widehat{H})\le n$; in particular, $\widehat H$ is torsion-free.
\end{lemma}

\begin{proof} If $A$ is any simple discrete $\widehat H$-module, then $A$ is finite (because $\widehat H$ is compact and its action on $A$ is continuous), so
$H^{n+1}(\widehat{H},A)\cong H^{n+1}(H,A) =\{0\}$ by cohomological goodness of $H$ and the assumption that ${\rm cd}(H)<n+1$.
Hence ${\cd}_p(\widehat H)\le n$ for every prime $p$ by \cite[Prop. 7.1.4]{RZ}, therefore
$${\rm cd}(\widehat H):=\sup \{ {\rm cd}_p(\widehat{H}) \mid p \mbox{ prime}\} \le n.$$

Finally, since ${\rm cd}_p(C) \le {\rm cd}_p(\widehat{H})<\infty$ for each prime $p$ and every closed subgroup $C \leqslant \widehat H$ (cf. \cite[Thm. 7.3.1]{RZ}),
and ${\rm cd}_p(\Z/p)=\infty$ we can conclude that
$\widehat H$ cannot contain subgroups of order $p$, for any prime $p$. Thus $\widehat H$ must be torsion-free, as claimed.
\end{proof}

\begin{prop} \label{prop:surj} Let $p$ be a prime and let $G$ be a residually finite cohomologically good group  such that ${\rm vcd}(G)<\infty$ and
$G$ contains finitely many conjugacy classes of subgroups (or, equivalently, elements) of order $p$. Then every element of order $p$ in the profinite completion $\widehat G$ is
conjugate to some element of $G$.
\end{prop}

\begin{proof} Arguing by contradiction suppose that there is some element $\gamma \in \widehat{G}$, of order $p$, such that $C=\langle \gamma \rangle$ is not conjugate to any subgroup of $G$.  By the assumptions, only finitely many conjugacy classes $\mathcal{C}_1,\dots,\mathcal{C}_k$, of subgroups of order $p$ in $\widehat G$,
intersect $G$ non-trivially.
Since each  $\mathcal{C}_i$, $i=1,\dots,k$, is a compact subset of $\widehat G$, avoiding the finite subgroup $C$,  there is a normal open subgroup $U$ of $\widehat G$
such that $CU \cap \mathcal{C}_i=\emptyset$ for every $i=1,\dots,k$. Since ${\rm vcd}(G)<\infty$, $G$ contains a normal torsion-free subgroup $K$ of finite index.
Then the closure $\overline{K}$, of $K$ in $\widehat G$, is naturally isomorphic to $\widehat K$, and hence it is torsion-free by Lemma \ref{lem:f-cd-t-f}
($K$ is cohomologically good by  \cite[Lemma 3.2.6]{RZ} and ${\rm cd(K)}={\rm vcd}(G)<\infty$).
So, after replacing $U$ by $U \cap \overline{K}$, we can assume that $U$ is torsion-free.

Now, $CU$ is an open subgroup of $\widehat G$, so $H=G \cap CU$ is a finite index subgroup of $G$, whose closure $\overline H$ in $\widehat G$ coincides with $CU$
(see \cite[Prop. 3.2.2]{RZ}). Since $H \cap \mathcal{C}_i=\emptyset$, $i=1,\dots,k$, and every subgroup of order $p$ in $G$ is contained in some $\mathcal{C}_i$,
we can conclude that $H$ has no elements of order $p$. On the other hand, since $CU$ is an extension of a torsion-free group $U$ by the cyclic group $C$, of order $p$,
we see that $CU$ cannot contain non-trivial elements of finite orders other than $p$. Recalling that $H \leqslant CU$, allows us to conclude that $H$ is torsion-free.

Since $|G:H|<\infty$ we can argue as in the case of $K$ above (using Lemma \ref{lem:f-cd-t-f}) to deduce that $\overline{H}=CU$ must be torsion-free. The latter
contradicts the fact that it contains $C$, completing the proof of the proposition.
 \end{proof}

\begin{cor}\label{cor:bij_cc} Suppose that $G$ is a finitely generated residually finite cohomologically good group with ${\rm vcd}(G)<\infty$. Then $G$ has finitely many conjugacy classes of subgroups of prime power order, and  the natural inclusion of $G$
in $\widehat G$ induces a bijection between the conjugacy classes of elements (or subgroups) of prime order in $G$ and in $\widehat G$.
\end{cor}

\begin{proof} By the assumptions, $G$ has a normal torsion-free finite index subgroup $H$. It follows that there can be only finitely many primes $p$ such that
$G$ contains some non-trivial $p$-subgroup. Let $p$ be such a prime. Since $G$ is cohomologically good, the same is true for $H$, so we can use a theorem of Weigel and the second author \cite[Thm. B]{WZ}
claiming that $H^n(H,\Z/p)$ is finite for every $n \ge 0$. Since $\Z/p$ is a field, the Universal Coefficient Theorem tells us that the $\Z/p$-vector space
$H^n(H,\Z/p)$ is the dual of $H_n(H,\Z/p)$, hence the latter is also finite. Therefore we can apply a result of Brown \cite[Lemma IX.13.2]{Brown} claiming that
$G$ contains finitely many conjugacy classes of $p$-subgroups.

Thus we can use Proposition \ref{prop:surj}, to conclude that the natural map between the conjugacy classes of elements of prime order in $G$ and in
$\widehat G$ is surjective.
This map is injective by Theorem \ref{thm:prime_order_cd}, so the corollary is proved.
\end{proof}

\begin{rem}
In the case when the group $G$ is virtually of type ${\rm FP}$, Thm. 8.2 in the survey paper \cite{Lochak} asserts (without proof)
that, with some extra work, a stronger version of Corollary \ref{cor:bij_cc}
can be derived from a general result of Symonds \cite[Thm.~1.1]{Symonds} (this was also
confirmed to us by Symonds in a private communication).
\end{rem}

An important tool for establishing cohomological goodness was discovered by Grune\-wald, Jaikin-Zapirain and the second author, and, independently, by Lorensen:

\begin{prop}[{\cite[Prop.~3.6]{GJZ}},{\cite[Cor. 3.11]{Lor}}]\label{prop:HNN-good} Let $G=H*_{B=A^t}$ be an HNN-extension of a cohomologically good group $H$,
where the associated subgroups $A$ and $B$ are also cohomologically good. Suppose that $G$ is residually finite, $H$, $A$ and $B$ are separable in $G$ and
the profinite topology on $G$ induces the full profinite topologies on $H$, $A$, and $B$. Then $G$ is cohomologically good.
\end{prop}

This allows us to show that in fact any group from the class $\AVR$ is cohomologically good.

\begin{prop} \label{prop:VCSH-good} Let $G \in \AVR$. Then $G$ is
residually finite, cohomologically good and has finite virtual cohomological dimension.
\end{prop}

\begin{proof}
By definition of the class $\AVR$, some finite index subgroup $H \leqslant G$ is a virtual retract of some right angled Artin group $A$.
Right angled Artin groups are residually finite
(see, for example, \cite[Ch. 3, Thm 1.1]{Droms}), hence $H$ and $G$ are both residually finite. The cohomological dimension $\cd(A)$, of $A$,
is equal to the clique number of the associated graph (this follows from the fact that $A$ acts freely and cocompactly on a CAT($0$)
cube complex of the appropriate dimension -- see \cite[Sec. 3.6]{Charney}), therefore
$\cd(H)\le \cd(A)<\infty$. Thus $\mathrm{vcd}(G)=\cd(H)<\infty$.

To show that $G$ is cohomologically good, we will first prove this for all right angled Artin groups (cf. \cite[Thm. 3.15]{Lor} and \cite{Lor-corrig}).
Let $B$ be a right angled Artin group corresponding to some finite simplicial graph $\Gamma$ with vertex set $V$. We will show that $B$ is
cohomologically good by induction on $|V|$. If $|V|=0$ then $B=\{1\}$ and the claim holds trivially.
Now, suppose that $|V|>0$ and choose any $S \subset V$ with $|V\setminus S|=1$.
Then $B$ splits as an HNN-extension of $B_S$ over another full subgroup $B_T$,
for some $T \subset S$ (see \cite[Sec. 7]{M-RAAG}).
Since $B_S$ and $B_T$ are a right angled Artin groups with less than $|V|$ generators, they are cohomologically good by the induction hypothesis.
Recall that both $B_T$ and $B_S$ are retracts of $B$ and $B$ is residually finite, therefore these subgroups are separable in $B$ and
the profinite topology of $B$ induces the full profinite topologies on these subgroups (cf. \cite[Lemma 3.1.5]{RZ}).
Hence $B$ is cohomologically good by Proposition~\ref{prop:HNN-good}.

Thus we have shown that any right angled Artin group is cohomologically good. Therefore, according to Lemma \ref{lem:retract-good},
the finite index subgroup $H \leqslant G$ is cohomologically good, as a virtual retract of $A$. Hence $G$ is itself cohomologically good by \cite[Lemma 3.2]{GJZ}.
\end{proof}

Combining Theorem \ref{thm:prime_order_cd} with Proposition \ref{prop:VCSH-good} and Lemma \ref{lem:hcs_sbgp} we immediately obtain the following statement:

\begin{cor}\label{cor:prime_index} Let $G$ be a virtually compact special group (or, more generally, let $G \in \AVR$). Then every element of prime order is
conjugacy distinguished in $G$.
\end{cor}

\section{Proof of the main result}\label{sec:proof}
Before proving the main result we will need two more auxiliary statements.

\begin{lemma}\label{lem:inf_order} Let $G\in \X$ and let $x \in G$ be an element of infinite order. Then $x$ is conjugacy distinguished in $G$.
\end{lemma}

\begin{proof} By Lemma \ref{lem:hcs_sbgp}, $G$ has a normal subgroup $H$, of some finite index $m \in \N$, such that $H$ is hereditarily conjugacy separable.
By the assumptions,  $x^m \in H$ is an infinite order element in the hyperbolic group $G$, so its centralizer $C_G(x^m)$ is virtually cyclic (cf. \cite[Prop. 3.5]{Mih}).
It follows that $C_G(x^m)$ is conjugacy separable. The second condition of Proposition \ref{prop:C-Z_crit} follows from Lemma   \ref{lem:hyp_in_X}.(b).
Therefore we can use this proposition to conclude that
$x$ is conjugacy distinguished in $G$, as required.
\end{proof}

\begin{cor}[cf. {\cite[Cor. 9.11]{M-RAAG}}]\label{cor:t-f_hcs} If $G \in \X$ and $H \leqslant G$ is a torsion-free subgroup of finite index, then $H$ is hereditarily conjugacy separable.
\end{cor}

\begin{proof} Note that $H \in \X$ by Remark \ref{rem:fi_in_X}, hence any element of infinite order is conjugacy distinguished in $H$ by Lemma \ref{lem:inf_order}.
Since $H$ is torsion-free, the only element of finite order in $H$, the identity element, must also be conjugacy distinguished.
Thus all elements of $H$ are conjugacy distinguished, i.e., $H$ is conjugacy separable.

Clearly the same argument applies to any finite index subgroup $K \leqslant H$. Therefore, $H$ is hereditarily conjugacy separable.
\end{proof}
\begin{proof}[Proof of Theorem \ref{thm:main}.] Consider any group $G \in \X$.
Choose a torsion-free normal subgroup $H\lhd G$ such that $n=|G:H|$ is minimal (such $H$ exists by Lemma \ref{lem:hcs_sbgp}).
We will prove the theorem by induction on $n$. If $n=1$ the statement holds because $H$ is hereditarily conjugacy separable by Corollary \ref{cor:t-f_hcs}.
So we can assume that $n>1$ and we have already established hereditary conjugacy separability for every group from
$\X$ which has a torsion-free normal subgroup of index less than $n$.

We will first show that $G$ is conjugacy separable. So, consider any element $x \in G$. If $x$ has infinite order, then $x$ is conjugacy distinguished in $G$ by Lemma \ref{lem:inf_order}. Thus we can suppose that $x$ has finite order.

Set $K=H \langle x \rangle$ and observe that $K \in \X$ by Remark \ref{rem:fi_in_X}.
If $|K:H|<n$ then $K$ is hereditarily conjugacy separable by the  induction hypothesis, so $x$ is conjugacy distinguished in $K$. But then Lemma \ref{lem:sep_in_fi_sbgp} implies that $x$ is conjugacy distinguished in $G$, as
$|G:K| \le |G:H|<\infty$.

Therefore we can assume that $|K:H|=n=|G:H|$. It follows that $G=K$, i.e., $G=H \langle x \rangle \cong H \rtimes \langle x \rangle$, as $H$ is torsion-free and $x$ has finite order (which must then be equal to $n$).
We will now consider two cases.

{\it Case 1:} $n=p$ is a prime number.  Then $x$ is conjugacy distinguished in $G$ by Corollary~\ref{cor:prime_index}.

{\it Case 2:} $n$ is a composite number. Thus $n=lm$ for some $l,m \in \N$, $1<l,m<n$.
We aim to use the criterion from Proposition \ref{prop:C-Z_crit}, so let's check that all of its assumptions are satisfied.

Let $F=H \langle x^m \rangle \leqslant G$. Then $F \in \X$ by Remark \ref{rem:fi_in_X} and $F \cong H \rtimes (\Z/l)$. Thus $F$ is hereditarily conjugacy separable by the induction hypothesis, as $|F:H|=l<n$. Evidently, $F\lhd G$ and $|G:F|=m$.
Every finite index subgroup of $C_G(x^m)$ is separable in $G$  by Lemma \ref{lem:hyp_in_X}.(b), so it remains to check that $x$ is conjugacy distinguished in $C_G(x^m)$.

Set $H_1=C_G(x^m) \cap H$, and observe that $C_G(x^m)=H_1 \langle x \rangle \cong H_1 \rtimes (\Z/n)$. Moreover, in view of Remark \ref{rem:fi_in_X},
$H_1 \in \X$ as $|C_G(x^m):H_1|=n<\infty$ and $C_G(x^m) \in \X$ by Lemma \ref{lem:hyp_in_X}.(a).

To verify that $x$ is conjugacy distinguished in $C_G(x^m)$, consider any element $y \in C_G(x^m)$ which is not conjugate to $x$ in $C_G(x^m)$.
Since $x^m$ is central in $C_G(x^m)$, we can let $L$ be the quotient of  $C_G(x^m)$ by $\langle x^m \rangle$, and let $\phi:C_G(x^m) \to L$ denote the natural epimorphism.

Clearly $\phi(H_1) \cong H_1$, as $H_1 \cap \ker \phi=\{1\}$.
Therefore $\phi(H_1)$ is torsion-free and $L=\phi(H_1) \langle
\phi(x) \rangle \cong H_1 \rtimes (\Z/m)$, implying that $L \in
\X$ (by Remark \ref{rem:fi_in_X}). Consequently, $L$ is
hereditarily conjugacy separable by the induction hypothesis, as
$|L:H_1|=m<n$.
Let us again consider two separate subcases.

{\it Subcase 2.1:} suppose that $\phi(x)$ and $\phi(y)$ are not conjugate in $L$. Then there is a finite group $M$ and a homomorphism
$\psi:L \to M$ such that $\psi(\phi(x))$ is not conjugate to $\psi(\phi(y))$ in $M$. Thus the homomorphism $\eta=\psi\circ \phi:C_G(x^m) \to M$ will distinguish the conjugacy classes of $x$ and $y$, as required.

{\it Subcase 2.2:} assume that $\phi(x)$ is conjugate to $\phi(y)$ in $L$. Since $\ker \phi \subseteq \langle x \rangle$, we can deduce that
there is $h \in C_G(x^m)$ such that $hyh^{-1}=z$, for some $z \in \langle x \rangle$.

Now, $z \neq x$, since we assumed that $y$ is not conjugate to $x$ in $C_G(x^m)$. Therefore $x=\xi(x) \neq \xi(z)=z$, where $\xi: C_G(x^m) \to \langle x \rangle$ is the natural retraction (coming from the decomposition of
$C_G(x^m)$ as a semidirect product of $H_1$ and $\langle x\rangle$). Recalling that $\langle x \rangle$ is abelian, we see that $\xi(y)=\xi(hyh^{-1})=\xi(z)$. Therefore $\xi(y)$ is not conjugate to $\xi(x)$ in the finite cyclic group
$\langle x \rangle$. Thus we have distinguished the conjugacy classes of $x$ and $y$ in this finite quotient of $C_G(x^m)$.

Subcases 2.1 and 2.2 together imply that $x$ is conjugacy distinguished in $C_G(x^m)$. Therefore we have verified all of
the assumptions of Proposition \ref{prop:C-Z_crit} (for $G$ and the finite index normal subgroup
$F \lhd G$), so we can apply this proposition to deduce that $x$ is conjugacy distinguished in $G$. Thus Case 2 is completed.

Cases 1 and 2 exhaust all possibilities, so we have established conjugacy separability for any group $G \in \X$, which possesses a torsion-free normal subgroup $H \lhd G$ of index $n$. If $K \leqslant G$ is any
subgroup of finite index, then $K \in \X$ by Remark \ref{rem:fi_in_X} and $H\cap K$ is a torsion-free normal subgroup in $K$ of index at most $n$.
So, either using the induction hypothesis (if $|K:(H\cap K)|<n$) or the above argument (if $|K:(H\cap K)|=n$), we can conclude that $K$ is conjugacy separable as well.
Hence $G$ is hereditarily conjugacy separable, and the step of induction has been established. This finishes the proof of the theorem.
\end{proof}

\end{document}